\documentclass[12pt]{article}
\usepackage[latin1]{inputenc}
\usepackage{amsmath}
\usepackage{amsfonts}
\usepackage{amssymb}
\usepackage{graphicx}
\usepackage{color}
\usepackage[english]{babel}
\usepackage{epsfig}
\usepackage{psfrag}
\usepackage{subfigure}

\newcommand{\R}{\ensuremath{\mathbb{R}}}
\newcommand{\E}{\ensuremath{\mathbb{E}}}
\newcommand{\W}[2]{\ensuremath{W^{(#1)}(#2)}}
\newcommand{\dW}[2]{\ensuremath{W^{(#1)\prime}(#2)}}
\newcommand{\Z}[2]{\ensuremath{Z^{(#1)}(#2)}}

\newtheorem{theorem}{Theorem}[section]

\newtheorem{cor}[theorem]{Corollary}

\newtheorem{thm}[theorem]{Theorem}
\newtheorem{rem}[theorem]{Remark}

\newtheorem{defi}[theorem]{Definition}

\renewcommand{\P}{\ensuremath{\mathbb{P}}}

\psfrag{x}{\tiny $x$}
\psfrag{s}{\tiny $s$}
\psfrag{t}{\tiny $t$}
\psfrag{t_0}{\tiny $T_0^-$}
\psfrag{t_1}{\tiny $\sigma_{a_1}$}
\psfrag{t_2}{\tiny $\sigma_{a_2}$}
\psfrag{g}{\tiny $\bar\gamma(a_1)$}
\psfrag{a_1}{\tiny $a_1$}
\psfrag{a_2}{\tiny $a_2$}

\title{\bf Spectrally negative L\'evy processes perturbed by functionals of their running supremum}
\author{{\large A. Kyprianou\footnote{{\sc University of Bath, UK.} E-mail: a.kyprianou@bath.ac.uk}
\ \ {\large and}
\  C. Ott\footnote{{\sc  University of Bath, UK.} E-mail: C.Ott@bath.ac.uk}}
}

\begin{document}
\renewcommand{\figurename}{Fig.}
\maketitle
\begin{abstract}\noindent 
In the setting of the classical Cram\'er-Lundberg risk insurance model, Albrecher and Hipp \cite{albr_hipp} introduced the idea of tax payments. More precisely, if $X = \{X_t : t\geq 0\}$ represents the Cram\'er-Lundberg process and, for all $t\geq 0$, $S_t  = \sup_{s\leq t}X_s$, then \cite{albr_hipp} study $X_t - \gamma S_t$, $t\geq 0$, where $\gamma\in(0,1)$ is the rate at which tax is paid. This model has been generalised to the setting that $X$ is a spectrally negative L\'evy process by Albrecher et al. \cite{albr_ren_zhou}. Finally Kyprianou and Zhou \cite{kyp_zhou} extend this model further by allowing the rate at which tax is paid with respect to the process $S = \{S_t : t\geq 0\}$ to vary as a function of the current value of $S$. Specifically, they consider the so-called {\it perturbed} spectrally negative L\'evy process,
\[
U_t=X_t-\int_{(0,t]}\gamma(S_u)\,{\rm d} S_u,\qquad t\geq 0,
\]
under the assumptions  $\gamma :[0,\infty)\rightarrow [0,1)$ and $\int_0^\infty (1-\gamma(s)){\rm d}s =\infty$.

In this article we show that a number of the identities in \cite{kyp_zhou} are still valid for a much more general class of rate functions $\gamma:[0,\infty)\rightarrow \mathbb{R}$. Moreover, we show that, with appropriately chosen $\gamma$, the perturbed process can pass continuously (ie. creep) into $(-\infty, 0)$ in two different ways.

\medskip

\noindent {\sc Key wor{\rm d}s and phrases}: Spectrally negative L\'evy process, excursion theory, creeping, ruin.

\medskip

\end{abstract}

\section{Introduction}
Let $X=\{X_t:t\geq 0\}$ be a spectrally negative L\'evy process defined on a filtered probability space $(\Omega,\mathcal{F},\mathbb{F}=\{\mathcal{F}_t\}_{t\geq 0},\P)$ satisfying the natural conditions (cf. p.39, Section 1.3 of~\cite{bichteler}). That is to say, a one-dimensional process which has stationary and independent increments, c\`adl\`ag paths with only negative discontinuities but which does not have monotone paths. For $x\in\R$, denote by $\P_x$ the probability measure under which $X$ starts at $x$ and write $\P_0=\P$. It is well known that a spectrally negative L\'evy process $X$ is characterised by its L\'evy triplet $(\gamma,\sigma,\Pi)$, where $\sigma\geq0, \gamma\in\R$ and $\Pi$ is a measure on $(-\infty,0)$ satisfying the condition $\int_{(-\infty,0)}(1\wedge x^2)\,\Pi(dx)<\infty$. By the L\'evy-It\^o decomposition, $X$ may be represented in the form
\begin{equation}
X_t=\sigma B_t-\gamma t+X^{(1)}_t+X^{(2)}_t,\label{LevyItodecomposition1}
\end{equation}
where $\{B_t:t\geq 0\}$ is a standard Brownian motion, $\{X^{(1)}_t:t\geq 0\}$ is a compound Poisson process with discontinuities of magnitude bigger than or equal to one and $\{X_t^{(2)}:t\geq 0\}$ is a square integrable martingale with discontinuities of magnitude strictly smaller than one and the three processes are mutually independent. In particular, if $X$ is of bounded variation, the decomposition reduces to
\begin{equation}
X_t=\mathtt{d}t-\eta_t\label{LevyItodecomposition2}
\end{equation}
where $\mathtt{d}>0$ and $\{\eta_t:t\geq 0\}$ is a driftless subordinator. Further let
\begin{equation*}
\psi(\theta):=\log\E[e^{\theta X_1}],\qquad\theta\geq 0,
\end{equation*}
be the Laplace exponent of $X$ which is known to be a strictly convex and infinitely differentiable function on $[0,\infty)$. The asymptotic behavior of $X$ is characterised by $\psi^\prime(0+)$, so that $X$ drifts to $\pm\infty$ or oscillates according to whether $\pm\psi^\prime(0+)>0$ or $\psi^\prime(0+)=0$, respectively.\\
\indent Denote by $S=\{S_t:t\geq 0\}$ the running supremum, that is, $S_t=\sup_{s\leq t}X_s$ for each $t\geq 0$. We are interested in  perturbing  $X$ by some functional of its running supremum $S$. Motivated by the results in~\cite{kyp_zhou}, our primary object of study is given by $U = \{U_t : t\geq 0\}$, where
\begin{equation*}
U_t=X_t-\int_{(0,t]}\gamma(S_u)\,{\rm d}S_u, \qquad t\geq 0,
\end{equation*}
for some locally $S$-integrable function $\gamma:[0,\infty)\to\R$. Such processes have appeared in the context  of insurance risk models with tax, where $X$ plays the role of the so-called surplus process (the wealth of an insurance company) and  $\gamma$ characterises the rate at which tax is paid with respect to the running maximum. One may also think of the process $U$ as a spectrally negative L\'evy process perturbed by a functional of its maximum in the spirit of \cite{PW1997}. In the special case that $\gamma:[0,\infty)\to[0,1)$ and $\int_0^\infty(1-\gamma(s))\,{\rm d}s=\infty$ our process $U$ agrees with the process studied in~\cite{kyp_zhou}. Under the even stronger assumption that $\gamma$ is a constant in $(0,1)$, the resulting process has been considered in~\cite{albr_hipp} and~\cite{albr_ren_zhou}.
In the simple case that $\gamma=0$ we are back to the process $X$, the so-called L\'evy insurance risk process in the context of ruin theory. The main objective of this article is to show that all of the identities in \cite{kyp_zhou} carry over to the setting where $\gamma$ belongs to the general class of locally $S$-integrable functions. Moreover, we will show that, for some choices of $\gamma$ it is possible for the process $U$ to enter $(-\infty, 0)$ continuously in two different ways.

 The key observation which, with the help of excursion theory, lea{\rm d}s to all our results is that we may write $U$ in the form
\begin{equation}
U_t=A_t-(S_t-X_t),\qquad t\geq 0,
\label{pathdecomp}
\end{equation}
where the process $A=\{A_t:t\geq 0\}$ is given by
\begin{equation}
A_t:=S_t-\int_{(0,t]}\gamma(S_u)\,{\rm d}S_u,\qquad t\geq 0.
\label{A}
\end{equation}
Assuming that $X_0 = x$, one may write $A_t=\bar\gamma(S_t)$, where 
\begin{equation*}
\bar\gamma(s):=s-\int_x^s\gamma(y)\,{\rm d}y,\qquad s\geq x.
\end{equation*}
Note that $A$ is a process of bounded variation and accordingly we may think of ${\rm d}A_t$ as a signed measure whose support, say $\mathcal{A}$, is contained in the support of the measure ${\rm d}S$. Suppose now that $\mathcal{B}$ consists of the countable union of open intervals of time which correspond to the epochs that the process $S-X$ spen{\rm d}s away from zero. Then ${\mathcal A}\cap{\mathcal B} = \emptyset$. As a consequence we may interpret (\ref{pathdecomp}) as a path decomposition in which excursions of $X$ from its maximum (equivalently excursions of $S-X$ away from zero) are `hung' off the trajectory of $A$ between its increment times (see Fig.~\ref{picture_1} for a symbolic representation). A more detailed description of this excursion-theoretic decomposition will follow in due course.\\
\indent We conclude this section by introducing the so-called scale functions (cf.~\cite{kyprianou}) which will henceforth play an important role and are defined as follows. For each given $q\geq 0$, we have $\W{q}{x} = 0$ when $x < 0$, and otherwise on $[0,\infty)$, $W ^{(q)}$ is the unique right continuous function whose Laplace transform is
\begin{equation*}
\int_0^\infty e^{-\theta x}\W{q}{x}\,dx=\frac{1}{\psi(\theta)-q},\qquad\theta>\Phi(q),
\end{equation*}
where $\Phi(q)$ is the largest solution to the equation $\psi(\theta)=q$ (there are at most two). For notational convenience we will write $W^{(0)}=W$. It is shown in Lemma 2.3 of~\cite{KuzKypRiv} that, for any $q\geq 0$, $W^{(q)}$ is absolutely continuous with respect to Lebesgue measure and strictly increasing. If $X$ is of unbounded variation, it is additionally known that $W^{(q)}$ is continuously differentiable on $(0,\infty)$ (cf. Lemma 2.4 of~\cite{KuzKypRiv}). In either case we shall denote by $W^{(q)\prime}$ the associated density. Finally, the behavior of $W^{(q)}$ and its right-derivative, written $W^{(q)\prime}_+$, at zero are known. Specifically, for all $q\geq 0$, we have
\begin{equation}\label{continuityatorigin}
\W{q}{0+}=\begin{cases}\mathtt{d}^{-1},&\text{if $X$ is of bounded variation,}\\0,&\text{if $X$ is of unbounded variation.}
\end{cases}
\end{equation}
and
\begin{equation}\label{derivativeatorigin}
W_+^{(q)\prime}(0+)=\begin{cases}
\frac{q+\Pi(-\infty,0)}{\mathtt{d}^2},&\text{if $\sigma=0$ and $\Pi(-\infty,0)<\infty$,}\\
\frac{2}{\sigma^2},&\text{if $\sigma>0$ or $\Pi(-\infty,0)=\infty$,}\
\end{cases}
\end{equation}
where we understand the second case to be $+\infty$ when $\sigma=0$ (cf. Lemma 3.1 and 3.2 of~\cite{KuzKypRiv}).\\

\section{Results}
Let us introduce
\begin{equation*}
\sigma_a:=\inf\{t>0\,:\,S_t=a\}\quad\text{and}\quad T_0^-:=\inf\{t>0\,:\,U_t<0\},
\end{equation*}
where we use the  usual convention  $\inf\emptyset := \infty$.

\begin{figure}[h]
\centering
\includegraphics{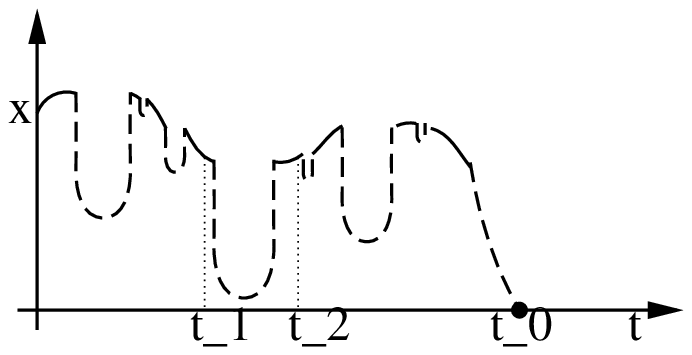}\quad\quad
\includegraphics{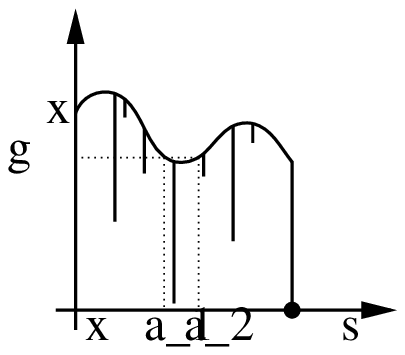}
\caption{A symbolic realisation of the trajectory of $U$ up to the moment it first enters $(-\infty,0)$, with its corresponding trace in the $(x,s)$-plane. The dashed u-shaped curves are schematic representations of  excursions of $X$ away from its maximum or, put differently, excursion of $U$ away from $A$.}\label{picture_1}
\end{figure}

\begin{thm}[One-sided and two-sided exit formulae]\label{two_sided_exit}
{\renewcommand{\theenumi}{\alph{enumi}}
\renewcommand{\labelenumi}{(\theenumi)}
Let $x>0$ be given and define $a^*(x):=\inf\{s\in[x,\infty):\bar\gamma(s)<0\}\in (x,\infty]$. Then, for any $q\geq  0$ and $x\leq a< a^*(x)$, we have
\begin{enumerate}
\item\label{equality_1} 
\begin{equation}
\E_x\big[e^{-q\sigma_a}1_{\{\sigma_a<T_0^-\}}\big]=\exp\bigg(-\int_x^{a}\frac{\dW{q}{\bar\gamma(s)}}{\W{q}{\bar\gamma(s)}}\,{\rm d}s\bigg),
\label{firstformula}
\end{equation}
\item\label{equality_2} \begin{equation*}
\E_x\big[e^{-qT_0^-}1_{\{T_0^-<\sigma_a\}}\big]=\int_x^{a}\exp\bigg(-\int_x^t\frac{\dW{q}{\bar\gamma(s)}}{\W{q}{\bar\gamma(s)}}\,{\rm d}s\bigg)f(\bar\gamma(t))\,{\rm d}t,
\end{equation*}
where $$f(z)=\frac{\Z{q}{z}\dW{q}{z}}{\W{q}{z}}-q\W{q}{z}$$
and $Z^{(q)}(x) = 1 + q\int_0^x W^{(q)}(y){\rm d}y$.
\item\label{equality_3} Suppose that $a^*(x)=\infty$. Then 
\begin{equation}\label{eq_1}
\E_x\big[e^{-qT_0^-}1_{\{T_0^-<\infty\}}\big]=\int_x^{\infty}\exp\bigg(-\int_x^t\frac{\dW{q}{\bar\gamma(s)}}{\W{q}{\bar\gamma(s)}}\,{\rm d}s\bigg)f(\bar\gamma(t))\,{\rm d}t.
\end{equation}
In particular,
\begin{equation}\label{eq_2}
\P_x[T_0^-<\infty]=1-\exp\bigg(-\int_x^\infty\frac{W^\prime(\bar\gamma(s))}{W(\bar\gamma(s))}\,{\rm d}s\bigg).
\end{equation}
\end{enumerate}}
\end{thm}

\begin{rem}
If we assume that $\gamma:[0,\infty)\to[0,1)$ with $\int_0^\infty(1-\gamma(s))\,{\rm d}s=\infty$, then $\bar\gamma$ is continuous, strictly increasing and has a well-defined inverse on $[x,\infty)$ which we shall denote by  $\bar\gamma^{-1}$. Then, for $a\geq x$, if we write $T_a^+=\inf\{t>0:U_t>a\} = \sigma_{\bar\gamma^{-1}(a)}$, Theorem~\ref{two_sided_exit} reads
\begin{eqnarray}
\E_x\big[e^{-qT_a^+}1_{\{T_a^+<T_0^-\}}\big]&=&\exp\bigg(-\int_x^{\bar\gamma^{-1}(a)}\frac{\dW{q}{\bar\gamma(s)}}{\W{q}{\bar\gamma(s)}}\,{\rm d}s\bigg)\notag\\
&=&\exp\bigg(-\int_x^a\frac{\dW{q}{y}}{\W{q}{y}(1-\gamma(\bar\gamma^{-1}(y)))}\,{\rm d}y\bigg)
\label{alsofordecreasing}
\end{eqnarray}
which agrees with Theorem 1.1 in~\cite{kyp_zhou}. Similarly, if $\gamma=0$, then $U_t=X_t,\tau^+_a:=\inf\{t>0:X_t>a\}$ and Theorem~\ref{two_sided_exit} reduces to
\begin{equation*}
\E_x\big[e^{-q\tau^+_a}1_{\{\tau^+_a<T_0^-\}}\big]=\exp\bigg(-\int_x^{a}\frac{\dW{q}{s}}{\W{q}{s}}\,{\rm d}s\bigg)\\
=\frac{\W{q}{x}}{\W{q}{a}}
\end{equation*}
and 
\begin{eqnarray*}
\E_x\big[e^{-qT_0^-}1_{\{T_0^-<\tau^+_a\}}\big]&=&\int_x^{a}\frac{\W{q}{x}}{\W{q}{t}}f(t)\,{\rm d}t\\
&=&-\W{q}{x}\int_x^{a}\bigg(\frac{Z^{(q)}}{W^{(q)}}\bigg)^\prime(t)\,{\rm d}t\\
&=&\Z{q}{x}-\W{q}{x}\frac{\Z{q}{a}}{\W{q}{a}},
\end{eqnarray*}
where $\Z{q}{x}:=1+q\int_0^x\W{q}{y}\,{\rm d}y$. This agrees with equations (8.8) and (8.9) of~\cite{kyprianou}. Also, by a straightforward calculation one sees that equations~\eqref{eq_1} and~\eqref{eq_2} reduce to equations (8.6) und (8.7) of~\cite{kyprianou}.
\end{rem}

\begin{rem}\label{linear_tax}
Fix $x>0$ and suppose $\gamma(s)\equiv\gamma\in(1,\infty)$. It follows that $\bar\gamma(s)=s(1-\gamma)+\gamma x$ and $a^*(x)=\frac{\gamma x}{\gamma-1}$. Then, for $q\geq 0$ and $x\leq a<a^*(x)$, the expression in~\eqref{firstformula} simplifies to
\begin{eqnarray*}
\E_x\big[e^{-q\sigma_a}1_{\{\sigma_a<T_0^-\}}\big]&=&\exp\bigg(\frac{1}{1-\gamma}\int_{a(1-\gamma)+\gamma x}^x\frac{\dW{q}{u}}{\W{q}{u}}{\rm d}u\bigg)\\
&=&\bigg(\frac{\W{q}{a(1-\gamma)+\gamma x}}{\W{q}{x}}\bigg)^{\frac{1}{\gamma-1}}.
\end{eqnarray*}
Moreover, if $\gamma(s)\equiv \gamma\in(0,1)$, one may recover the first formula of Remark 1.1 in~\cite{kyp_zhou} by a similar computation or an application of~\eqref{alsofordecreasing}.
\end{rem}

\begin{rem}
Let $x>0$ be given and assume that $X$ drifts to $+\infty$ or, equivalently, that $\psi^\prime(0+)>0$. Moreover, suppose that $\gamma(s)\equiv\gamma\in(0,1)$ and hence $\bar\gamma(s)=s(1-\gamma)+\gamma x$. Then, using the fact that $\lim_{s\to\infty}W(s)=1/\psi^\prime(0+)$ (cf. Lemma 3.3~in~\cite{KuzKypRiv}), it follows from~\eqref{eq_2} that
\begin{equation*}
\P_x[T^-_0<\infty]=1-\exp\bigg(-\frac{1}{1-\gamma}\int_x^\infty\frac{W^\prime(s)}{W(s)}{\rm d}s\bigg)=1-(\psi^\prime(0+)W(x))^{\frac{1}{1-\gamma}}.
\end{equation*}
This is analogous to equation (8.7) in~\cite{kyprianou}.
\end{rem}

The proof of Theorem \ref{two_sided_exit} makes heavy use of excursion theory for the process $S-X$. We refer the reader to~\cite{bertoin_book}, Chapters 6 and 7 for background reading. We shall spend a moment setting up some necessary notation which will be used throughout the remainder of the paper. The process $L_t:=S_t-x$ serves as local time at $0$ for the Markov process $S-X$ under $\P_x$. Write $L^{-1}:=\{L^{-1}_t:t\geq 0\}$ for the right-continuous inverse of $L$.
The Poisson point process of excursion indexed by local time shall be denoted by $\{(t,\epsilon_t):t\geq 0\}$, where
\begin{equation*}
\epsilon_t=\{\epsilon_t(s):=X_{L^{-1}_t}-X_{L^{-1}_{t-}+s}:0<s< L^{-1}_t-L^{-1}_{t-}\}
\end{equation*}
whenever $L^{-1}_t-L^{-1}_{t-}>0$. Accordingly, we refer to a generic excursion as $\epsilon(\cdot)$ (or just $\epsilon$ for short) belonging to the space $\mathcal{E}$ of canonical excursions. The intensity measure of the process $\{(t,\epsilon_t):t\geq 0\}$ is given by ${\rm d}t\times {\rm d}n$, where $n$ is a measure on the space of excursions (the excursion measure). A functional of the canonical excursion which will be of interest is $\overline\epsilon=\sup_{s<\zeta}\epsilon(s)$, where $\zeta(\epsilon)=\zeta$ is the length of an excursion. A useful formula for this functional that we shall make use of is the following (cf.~\cite{kyprianou}, Equation (8.18)):
\begin{equation}
n(\overline\epsilon>x)=\frac{W'(x)}{W(x)},
\label{property_ppp_tail}
\end{equation}
provided that $x$ is not a discontinuity point in the derivative of $W$, which is only a concern when $X$ is of bounded variation, in which case there are at most countably many such discontinuities. Another functional of $\epsilon$ that we will also use is $\rho_k:=\inf\{s>0:\epsilon(s)>k\}$, the first passage time above $k$ of the canonical excursion $\epsilon$. Note that, for $a\geq x$, it follows that, under $\P_x$, the event that $S_t=a$ coincides with the event that the process $S_t$ climbs from $x$ to $a$ for the first time. Consequently, $L^{-1}_{a-x}=\tau^+_{a}$.\\

\noindent\textbf{Proof of Theorem~\ref{two_sided_exit}:}
\mbox{ }\\
\indent \eqref{equality_1} For $a\geq x$ we have
\begin{equation*}
\{\sigma_a<T_0^-\}=\{\bar\epsilon_s\leq\bar\gamma(x+s)\text{ for all }0\leq s\leq a-x\}.
\end{equation*}
Recall that for each $q\geq 0$, we have the exponential change of measure 
\[
\left.\frac{{\rm d}\P^{\Phi(q)}}{{\rm d}\P}\right|_{\{X_s:s\leq t\}} = e^{\Phi(q) X_t - qt}, \qquad t\geq 0.
\]
Then, recalling that for each $t\geq 0$, $L^{-1}_t$ is a stopping time, we have for $x>0$,
\begin{eqnarray}
\E_x\big[e^{-q\sigma_a}1_{\{\sigma_a<\tau_0^-\}}\big]&=&\E_x\big[e^{-qL^{-1}_{a-x}}1_{\{\overline \epsilon_s\leq\bar\gamma(x+s)\text{ for all }0\leq s\leq a-x\}}\big]\notag\\
&=&e^{-(a-x)\Phi(q)}\E^{\Phi(q)}_x\big[1_{\{\overline \epsilon_s\leq\bar\gamma(x+s)\text{ for all }0\leq s\leq a-x\}}\big]\notag\\
&=&e^{-(a-x)\Phi(q)}\exp\bigg(-\int_0^{a-x}n_{\Phi(q)}(\overline\epsilon>\bar\gamma(x+s)\,{\rm d}s)\bigg)\notag\\
&=&\exp\bigg(-\int_0^{a-x}\frac{\dW{q}{\bar\gamma(x+s)}}{\W{q}{\bar\gamma(x+s)}}\,{\rm d}s\bigg).\label{chvar}
\end{eqnarray}
Here, $n_{\Phi(q)}$ is the excursion measure of $S-X$ under $\mathbb{P}^{\Phi(q)}$, which is  known to satisfy 
\[
n_{\Phi(q)}(\bar\epsilon>x) = \frac{W^{(q)\prime}(x)}{W^{(q)}(x)} - \Phi(q);
\]
see for example formula (2.7) of \cite{kyp_zhou}.
Now changing variables in (\ref{chvar}) gives (\ref{firstformula}).\\
\indent \eqref{equality_2} An application of the compensation formula yiel{\rm d}s
\begin{eqnarray*}
\lefteqn{\E_x\big[e^{-q\tau_0^-}1_{\{\tau_0^-<\sigma_a\}}\big]}\\
&&=\E_x\bigg[\sum_{0<t\leq a-x}e^{-qL^{-1}_{t-}-q\rho_{\bar\gamma(t+x)}(\epsilon_t)}1_{\{\overline\epsilon_s\leq\bar\gamma(s+x)\forall s<t\}}1_{\{\overline\epsilon_t>\bar\gamma(t+x)\}}\bigg]\\
&&=\E_x\bigg[\int_0^{a-x}e^{-qL_t^{-1}}1_{\{\overline\epsilon_s\leq\bar\gamma(s+x)\forall s<t\}}\int_{\mathcal{E}}e^{-q\rho_{\gamma(t+x)}(\epsilon)}1_{\{\overline\epsilon>\bar\gamma(t+x)\}}n({\rm d}\epsilon)\,{\rm d}t\bigg]\\
&&=\E_x\bigg[\int_0^{a-x}e^{-qL_t^{-1}}1_{\{\overline\epsilon_s\leq\bar\gamma(s+x)\forall s<t\}}f(\bar\gamma(t+x))\,{\rm d}t\bigg]\\
&&=\int_0^{a-x}e^{-t\Phi(q)}\E_x^{\Phi(q)}\big[1_{\{\overline\epsilon_s\leq\bar\gamma(s+x)\forall s<t\}}\big]f(\bar\gamma(t+x))\,{\rm d}t\\
&&=\int_0^{a-x}e^{-t\Phi(q)}\exp\bigg(-\int_0^tn_{\Phi(q)}(\overline\epsilon>\bar\gamma(s+x))\,{\rm d}s\bigg)f(\bar\gamma(t+x))\,{\rm d}t\\
&&=\int_0^{a-x}\exp\bigg(-\int_0^t\frac{\dW{q}{\bar\gamma(s+x)}}{\W{q}{\bar\gamma(s+x)}}\,{\rm d}s\bigg)f(\bar\gamma(t+x))\,{\rm d}t,
\end{eqnarray*}
where in the first equality the time index runs over local times and the sum is the usual shorthand for integration with respect to the Poisson counting measure of excursions, and
\begin{equation*}
f(z)=\int_{\mathcal{E}}e^{-q\rho_z(\epsilon)}1_{\{\overline\epsilon>z\}}n({\rm d}\epsilon)=\frac{\Z{q}{z}\dW{q}{z}}{\W{q}{z}}-q\W{q}{z}
\end{equation*}
is an expression taken from Theorem 1 (equation (18)) of~\cite{exitproblems}. The proof is completed by a straightforward change of variables.\\
\indent\eqref{equality_3} The first part follows by letting $a\to\infty$ in~\eqref{equality_2} and the second part by looking at the complement and then using a similar argument as in~\eqref{equality_1}.\hfill$\Box$

\section{Creeping}

\noindent In principle there are two ways for $U$ to enter $(-\infty, 0)$ continuously; either it goes below zero by creeping  during an excursion away from the curve $\bar\gamma$ or it creeps over zero whilst moving along the curve $\bar\gamma$ at the moment that $\bar\gamma = 0$ (see Fig.~\ref{picture_2}).
\begin{figure}[h]
\centering
\includegraphics{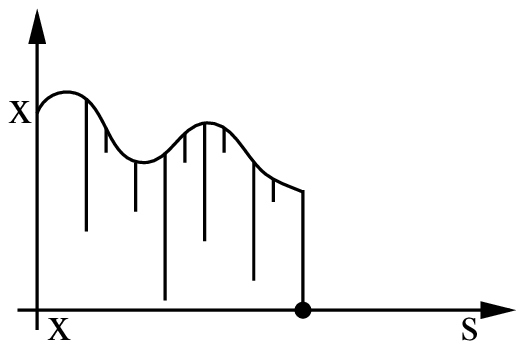}\quad\quad
\includegraphics{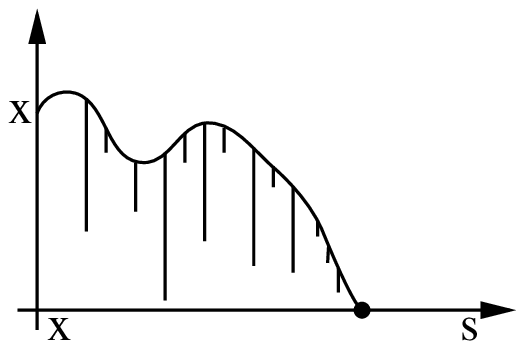}
\caption{Two different ways for $U$ to make first passage into $(-\infty,0)$.}\label{picture_2}
\end{figure}
This lea{\rm d}s to the next definition.
\begin{defi} We say that the process $U$ exhibits type I creeping under $\P_x$ if 
$\P_x(U_{T^-_0} = 0; \sigma_{a^*(x)} >T^-_0)>0$.
We say that the process $U$ exhibits type II creeping under $\P_x$ if $a^*(x)<\infty$ and $\P_x[\sigma_{a^*(x)}=T_0^-]>0$. Note that necessarily $U_{T^-_0 }= 0$ on the event $\{\sigma_{a^*(x)} = T^-_0\}$.
\end{defi}
In Section \ref{additional} we shall make some remarks regarding type I creeping. However, for the present, let us dwell on type II creeping, for which there exists an integral test.
\begin{cor}
Fix $x>0$ and recall that  $a^*(x)=\inf\{s\in[x,\infty):\bar\gamma(s)<0\}$. Assume $a^*(x)<\infty$. We have for all $q\geq 0$,
\begin{equation*}
\E_x\big[e^{-qT_0^-}1_{\{T_0^-=\sigma_{a^*(x)}\}}\big]=\exp\bigg(-\int_x^{a^*(x)}\frac{\dW{q}{\bar\gamma(s)}}{\W{q}{\bar\gamma(s)}}\,{\rm d}s\bigg).
\end{equation*}
\end{cor}

If $X$ is a compound Poisson process, then the time $X$ spends at the maximum has strictly positive Lebesgue measure and hence one would intuitively expect that type II creeping occurs. In fact, under some assumptions on the behaviour of $\bar\gamma$, it turns out that only spectrally negative L\'evy processes of bounded variation possess the type II creeping property.

\begin{cor}
Fix $x>0$ and assume that $\gamma:[0,\infty)\to(1,\infty)$ is continuous. Further suppose that $a^*(x)<\infty$. Then $X$ exhibits type II creeping under $\P_x$ if and only if $X$ is of bounded variation.
\end{cor}

\begin{proof}First observe that the assumptions on $\gamma$ imply that $\bar\gamma:(x,a^*(x))\to(x,0)$ is a continuously differentiable bijection. Further let $C_1:=\min_{0\leq s\leq a^*(x)} \vert 1-\gamma(s)\vert>0$ and $C_2:=\max_{0\leq s\leq a^*(x)}\vert1-\gamma(s)\vert<\infty$.\\
\indent If $X$ is of bounded variation, and hence takes the form (\ref{LevyItodecomposition2}), we have by a change of variables
\begin{equation*}
\int_x^{a^*(x)}\frac{W^\prime(\bar\gamma(s))}{W(\bar\gamma(s))}\,{\rm d}s=\int_0^x\frac{W^\prime(t)}{W(t)\vert1-\gamma(\bar\gamma^{-1}(t))\vert}\,{\rm d}t\leq\frac{\mathtt{d}}{C_1}(W(x)-\mathtt{d}^{-1})
\end{equation*}
and hence type II creeping follows. On the other hand, if $X$ is of unbounded variation, it follows similarly that
\begin{equation*}
\int_x^{a^*(x)}\frac{W^\prime(\bar\gamma(s))}{W(\bar\gamma(s))}\,{\rm d}s\geq\frac{1}{C_2}\int_0^x\log(W)^\prime(t)\,{\rm d}t\geq \frac{1}{C_2}\big[\log(W(x))-\log(W(0+))\big].\end{equation*}
The last expression equals infinity since $W(0+)=0$ and, consequently, type II creeping cannot occur.\hfill$\square$
\end{proof}\\

We conclude this section with an example of type II creeping for a process $X$ which includes a Gaussian component $\sigma>0$ in the case that  $\bar\gamma$ has infinite gradient when hitting zero. This shows that relaxing the conditions on $\gamma$ can lead to type II creeping in the unbounded variation case. To this end, we need some auxiliary quantities. Let $a>0$ be fixed and define, for $y\in[0,a]$,
\begin{equation*}
f(y):=y-(a-y)^{\frac{1}{2}}.
\end{equation*}
Clearly $f(0)<0$ and $f(a)>0$. Since $f$ is strictly increasing on $[0,a]$ and continuous, by the Intermediate Value Theorem, there exists a unique $x^*\in(0,a)$ such that $f(x^*)=0$. Now let $\gamma\in(1,\infty)$ and define for $s\in[0,\infty)$,
\begin{equation*}
\gamma(s):=\begin{cases}1+\frac{1}{2}(a-s)^{-\frac{1}{2}},&s\leq a,\\
\gamma,&s>a.\end{cases}
\end{equation*}
Hence, using the definition of $x^*$, we see that
\begin{equation*}
\bar\gamma(s)=\begin{cases}(a-s)^{\frac{1}{2}},&s\leq a\\
(1-\gamma)(s-a),&s>a.\end{cases}\end{equation*}
In particular, $a^*(x^*)=a$. Changing variables and using the fact that $\sigma>0$ (which implies that $W\in C^1(0,\infty)$ and $\lim_{u\downarrow 0}uW(u)^{-1}=\sigma^2/2$) yields
\begin{equation*}
\int_{x^*}^{a}\frac{W^\prime(\bar\gamma(s))}{W(\bar\gamma(s))}\,ds=2\int_0^{x^*}\frac{uW^\prime(u)}{W(u)}\,{\rm d}u\leq 2x^* \sup_{0<u\leq x^*}\frac{uW^\prime(u)}{W(u)}<\infty.
\end{equation*}
Hence, type II creeping occurs under $\mathbb{P}_{x^*}$.

\section{Additional results relevant to risk theory}\label{additional}

Let us return to the setting of the stochastic perturbation $U$ in the setting of insurance risk.
It is also possible to obtain the analogous statements to Theorem 1.2 and 1.3 in \cite{kyp_zhou}. 
The analogue of the first of these two theorems concerns the expectation of a path functional which can be interpreted as the net present value of tax paid until ruin and reads as follows.

\begin{thm}\label{npv}
Let $x>0$ and recall $a^*(x)=\inf\{s\in[x,\infty):\bar\gamma(s)<0\}\in(0,\infty]$. For $q\geq 0$ we have
\begin{equation*}
\E_x\bigg[\int_0^{T_0^-}e^{-qu}\gamma(S_u)\,dS_u\bigg]=\int_x^{a^*(x)}\exp\bigg(-\int_x^t\frac{\dW{q}{\bar\gamma(s)}}{\W{q}{\bar\gamma(s)}}\,{\rm d}s\bigg)\gamma(t)\,{\rm d}t.
\end{equation*}
\end{thm}

\begin{rem}
Fix $x>0$ and suppose $\gamma(s)=\gamma\in(1,\infty)$. A computation as in Remark~\ref{linear_tax} shows that
\begin{equation*}
\E_x\bigg[\int_0^{T_0^-}e^{-qu}\gamma(S_u)\,dS_u\bigg]=\frac{\gamma}{\gamma-1}\int_0^x\bigg(\frac{\W{q}{t}}{\W{q}{x}}\bigg)^{\frac{1}{\gamma-1}}{\rm d}t.
\end{equation*}
Similarly, if one assumes that $\gamma(s)\equiv\gamma\in(0,1)$, it is straightforward to recover the second formula in Remark 1.1 of~\cite{kyp_zhou}.
\end{rem}

Although unnecessary, for the sake of presentational convenience, we shall restrict ourselves to the case that $\gamma:[0,\infty)\rightarrow(1,\infty)$ in order to state an analogue of Theorem 1.3 in \cite{kyp_zhou}. In that case, $\bar\gamma$ is a strictly decreasing function and accordingly has an inverse, $\bar\gamma^{-1}$. Note that in \cite{kyp_zhou} it was assumed that $\gamma:[0,\infty)\rightarrow[0,1)$ such that $\int_0^\infty (1- \gamma(s)){\rm d}s = \infty$.  If we refer to the latter as a {\it light tax regime} then we may think of the current setting as a {\it heavy tax regime}. We have the following result, the second part of which addresses the issue of type I creeping.

\begin{thm}\label{overshoot_undershoot}
Fix $x>0$ and suppose $a^*(x)<\infty$. Let $\kappa=L^{-1}_{L_{T_0^-}-}$, the last moment that tax is paid before ruin. Denote by $\nu$ the L\'evy measure of $-X$. For any $z>0$, $x>\theta\geq y\geq 0$ and $\alpha,\beta\geq 0$, we have
\begin{eqnarray*}\label{gs}
\lefteqn{
\mathbb{E}_x\left(
e^{-\alpha\kappa - \beta(T^-_0 - \kappa)}; A_{T^-_0}\in {\rm d}\theta, U_{T^-_0-}\in {\rm d}y, - U_{T^-_0}\in {\rm d}z
\right)
}&&\notag\\
&&=
\frac{1}{ \gamma(\bar\gamma^{-1}(\theta))-1}  \exp\left\{-\int^x_\theta
\frac{W^{(\alpha)\prime}(v)}{W^{(\alpha)}(v)(\gamma(\bar{\gamma}^{-1}(v))-1)}{\rm d}v \right\} \notag\\
&&\hspace{0.5cm}\cdot\bigg[
\left\{W^{(\beta)\prime}(\theta-y)-\frac{W^{(\beta)\prime}(\theta)}{W^{(\beta)}(\theta)}W^{(\beta)}(\theta-y)\right\}\nu(y+{\rm d}z) 
\mathbf{1}_{\{y<\theta\}}
{\rm d}y \notag\\
&&\hspace{5.5cm}+
 W^{(\beta)}(0+)\nu(\theta + {\rm d}z)\delta_\theta({\rm d}y) \bigg]{\rm d}\theta
\end{eqnarray*}
where $\delta_{\theta}({\rm d}y)$ is the Dirac measure which assigns unit mass to the point $\theta$. Furthermore, for $0<\theta<x$ we also have
\begin{eqnarray*}
\label{gs-creep}
\lefteqn{\mathbb{E}_x\left(
e^{-\alpha\kappa - \beta(T^-_0 - \kappa)}; A_{T^-_0}\in {\rm d}\theta, U_{T^-_0} =0\right) }\notag\\
&&=
\frac{1}{ \gamma(\bar\gamma^{-1}(\theta))-1}  \exp\left\{-\int^x_\theta
\frac{W^{(\alpha)\prime}(y)}{W^{(\alpha)}(y)(\gamma(\bar{\gamma}^{-1}(y))-1)}{\rm d}y \right\}\notag\\
&&\hspace{5cm}\cdot\frac{\sigma^2}{2} \left\{ \frac{W^{(\beta)\prime}(\theta)^2}{W^{(\beta)}(\theta)} - W^{(\beta)\prime\prime}(\theta)\right\}{\rm d}\theta,
\end{eqnarray*}
where $\sigma$ is the Gaussian coefficient in the L\'evy-It\^{o} decomposition.
\end{thm}

The proof of both of these theorems is virtually identical to the proofs of Theorems 1.2 and 1.3 in \cite{kyp_zhou} once the obvious adjustments have been made and accordingly are left as an exercise to the reader.

\end{document}